\documentclass[a4paper,14pt]{article}
\def\ds{\displaystyle}
\def\ni{\noindent}
\usepackage[a4paper, total={6in, 8in}]{geometry}
\usepackage{setspace}
\usepackage{amssymb}
\usepackage[utf8]{inputenc}
\usepackage[english]{babel}
\usepackage{amsthm}
\usepackage{amsmath}
\allowdisplaybreaks[1]

\newtheorem{theorem}{Theorem} [section]
\newtheorem{lemma}{Lemma} [section]

\numberwithin{equation}{section}

\begin{document}
\begin{center}
{\bf \Large Some congruences for $(s,t)$-regular bipartitions modulo $t$}
\end{center}
\begin{center}
\bf T. Kathiravan\footnote[1]{Email: kkathiravan98@gmail.com}
\bf and
\bf K. Srilakshmi\footnote[2]{srilakshmi@iisertvm.ac.in }
\end{center}
{\baselineskip .5cm \begin{center}\footnotemark[1]\footnotemark[2] School of Mathematics,\\ Indian Institute of Science Education and Research,\\ Thiruvananthapuram-695 551, Kerala,\\India.\\
\end{center}}
\begin{center} \footnotemark[1]The Institute of Mathematical Sciences,\\ CIT Campus, Taramani,\\ Chennai 600113, India\\
\end{center}
\ni {\bf \small Abstract:}
\par In this work, we study the function $B_{s,t}(n)$, which counts the number of $(s,t)$-regular bipartitions of $n$. 
Recently, many authors proved infinite families of congruences modulo $11$ for $B_{3,11}(n)$, modulo $3$ for $B_{3,s}(n)$ and modulo $5$ for $B_{5,s}(n)$.
Very recently, Kathiravan proved several infinite families of congruences modulo $11$, $13$ and $17$ for $B_{5,11}(n)$, $B_{5,13}(n)$ and $B_{81,17}(n)$. 
In this paper, we will prove infinite families of congruences modulo $5$ for $B_{2,15}(n)$, modulo $11$ for $B_{7,11}(n)$, modulo $11$  for $B_{27,11}(n)$ and modulo $17$ for $B_{243,17}(n)$ . \\

\ni {\bf \small 2010 Mathematics Subject Classification:} 11P83, 05A17.\\
\ni {\bf \small Keywords:} Congruence, Regular Bipartition.

\section{Introduction}
For $n$ a positive integer, a partition of $n$ is a non-increasing sequence of positive integers whose sum is $n$. The number of partitions of $n$ is denoted by $p(n)$. The generating function for $p(n)$, is given by
\begin{equation}\label{s1}
\sum_{n=0}^{\infty}p(n)q^n=\frac{1}{(q;q)_\infty}.
\end{equation}
For a nonzero integer $k$, we define the general partition function $p_k(n)$ as the coefficient of $q^n$ in the expansion of $(q;q)^k_\infty$. If $k=-1$ we have usual partition function $p(n)$.
The generating function for $p_k(n)$, is given by
\begin{equation}\label{1w}
\ds\sum_{n=0}^{\infty}p_k(n)q^n=(q;q)^k_\infty,
\end{equation}
where as customary, we define
\begin{equation}\nonumber
f_k:=(q^k;q^k)_\infty=\prod^\infty_{m=1}(1-q^{mk}).
\end{equation}
In \cite{Ramanujan,ramanujan2000collected} Ramanujan obtained the beautiful identities
\begin{equation}\label{j1}
\ds\sum_{n=0}^{\infty}p(5n+4)q^n=5\frac{f^5_5}{f^6_1}
\end{equation}
and
\begin{equation}\label{j2}
\ds\sum_{n=0}^{\infty}p(7n+5)q^n=7\frac{f^3_7}{f^4_1}+49q\frac{f^7_7}{f^8_1}.
\end{equation}
Ramanujan \cite{Ramanujan} give a brief of the identities \eqref{j1}, he did not prove the identities \eqref{j2} in \cite{Ramanujan}, but \cite{Ramanujan1} he did give a sketch of his proof of identities \eqref{j2} in his unpublished manuscript of the partition and $\tau$- function. Note that \eqref{j1} and \eqref{j2} immediately yield the congruences $p(5n+4)\equiv0\pmod5$ and $p(7n+5)\equiv0\pmod7.$\\
Ramanujan partition congruences motivated an investigation of many classes of partitions, such as $\ell$-regular partitions. For a positive integer $\ell$, a partition is said to be $\ell$-regular if none of its parts is divisible by $\ell$. If $b_\ell(n)$ denote the number of $\ell$-regular partitions of $n$, then the generating function for $b_\ell(n)$, is given by
 \begin{equation}\label{0.a}
\sum_{n=0}^{\infty}b_{\ell}(n)q^n=\frac{f_\ell}{f_1}.
\end{equation}
In recent years, many authors studied arithmetic properties of $\ell$-regular partitions \cite{Carlson, Cui, Dandurand, Furcy, Gordon, Hirschhorn, Webb, Xia, Yao}.\\
Recall that, for a positive integers $s>1$ and $t>1$, a bipartition $(\lambda,\mu)$ of $n$ is a pair of partitions $(\lambda,\mu)$ such that the sum of all the parts equals $n$. A $(s,t)$-regular bipartition of $n$ is a bipartition $(\lambda,\mu)$ of $n$ such that $\lambda$ is a $s$-regular partition and $\mu$ is a $t$-regular partition. If $B_{s,t}(n)$ denote the number of $(s,t)$-regular bipartitions of $n$, then the generating function $B_{s,t}(n)$, is given by
\begin{equation}\label{0.1}
\sum_{n=0}^{\infty}B_{s,t}(n)q^n=\frac{f_sf_t}{f^2_1}.
\end{equation}
Recently, Lin \cite{Lin1} proved infinite families of congruence modulo $3$ for $B_7(n)$, by using Ramanujan's two modular equation of degree $7$, and in \cite{Lin2}, he proved infinite families of congruence modulo $3$ for $B_{13}(n)$. For more related works, see \cite{Kathiravan, Wang}.\\
Very recently ,\cite{Dou} Dou proved that, for $n\geq0$ and $\alpha\geq2$,
\begin{equation*}
B_{3,11}\left(3^\alpha n+\frac{5\cdot3^{\alpha-1}-1}{2}\right)\equiv0\pmod{11}.
\end{equation*}
Adiga and  Ranganatha \cite{Adiaga} proved infinite families of congruences modulo $3$ for $B_{3,7}(n)$ and Xia and Yao \cite{Xia1} proved several infinite families of congruences modulo $3$ for $B_{3,s}(n)$, modulo $5$ for $B_{5,s}(n)$ and modulo $7$ for $B_{3,7}(n)$. For example,
let $s$ be a positive integer and let $p\geq5$ be a prime, for $n\geq0$,
\begin{equation*}
B_{3,s}\left(p^{2\alpha+1}n+\frac{(1+s)(p^{2\alpha+2}-1)}{24}\right)\equiv0\pmod3.
\end{equation*}
Very Recently, Kathiravan \cite{Kathiravan1} proved several infinite families of congruences modulo $11$, $13$ and $17$ for $B_{5,11}(n)$, $B_{5,13}(n)$ and $B_{81,17}(n)$. For example, for all $n\geq0$ and $m\geq0$,
\begin{equation*}
B_{5,13}\left(5^{6m+5}(5n+k)+\frac{5^{6m+5}-2}{3}\right)\equiv0\pmod{13},~~\textrm{where}~~k=1,5.
\end{equation*}
In this paper, we will prove several infinite families of congruences modulo $5$, $11$, $11$ and $17$ for $B_{2,15}(n)$, $B_{7,11}(n)$, $B_{27,11}(n)$ and $B_{243,17}(n)$. The main results of this paper are as follows,
\begin{theorem}\label{y}\normalfont For all $n\geq0$ and $m\geq0$,
\begin{eqnarray}
B_{2,15}\left(3^{2m+1}n+\frac{7\cdot3^{2m+1}-5}{8}\right)&\equiv&2^mB_{2,15}(3n+2)\pmod{5}\label{y1},\\
B_{2,15}\left(3^{2m+2}n+\frac{23\cdot3^{2m+1}-5}{8}\right)&\equiv&0\pmod{5}\label{y2},\\
B_{2,15}\left(3^{2(m+1)+1}n+\frac{13\cdot3^{2(m+1)}-5}{8}\right)&\equiv&0\pmod{5}\label{y3}.
\end{eqnarray}
\end{theorem}
\begin{theorem}\label{a}\normalfont For all $n\geq0$ and $m\geq0$,
\begin{eqnarray}
B_{7,11}\left(7^{12m}n+\frac{2\cdot7^{12m}-2}{3}\right)&\equiv&3^mB_{7,11}(n)\pmod{11}\label{a1},\\
B_{7,11}\left(7^{12m+11}(7n+k)+\frac{2\cdot7^{12m+11}-2}{3}\right)&\equiv&0 \pmod{11},~\textrm{where}~k=1,5,6.\label{a2}
\end{eqnarray}
\end{theorem}
\begin{theorem}\label{b}\normalfont For all $n\geq0$ and $m\geq4$,
\begin{eqnarray}
B_{27,11}\left(3^mn+\frac{5\cdot3^{m-1}-3}{2}\right)&\equiv&0 \pmod{11}.
\end{eqnarray}
\end{theorem}
\begin{theorem}\label{c}\normalfont For all $n\geq0$,
\begin{eqnarray}
B_{243,17}(81n+23)&\equiv&0\pmod{17},\\
B_{243,17}(81n+77)&\equiv&0\pmod{17}.
\end{eqnarray}
\end{theorem}
\section{The identities}
In this section, we prove some lemmas to prove our main results.
By the binomial theorem for any prime $p$,
\begin{equation}\label{k1}
f_p\equiv f^p_1\pmod p.
\end{equation}
\begin{lemma}\normalfont (Berndt \cite[p.49]{Berndt}), we have
\begin{eqnarray}\label{4w}
\frac{f_2^2}{f_1}&=&\frac{f_6 f_9^2}{f_3 f_{18}}+q\frac{f_{18}^2}{f_9}.
\end{eqnarray}
\end{lemma}
\begin{lemma}\normalfont (Hirschhorn and Sellers \cite{M. D. H and J. A. S}), we have
\begin{equation}\label{a5}
\frac{f_2}{f^2_1}=\frac{f^4_6f^6_9}{f^8_3f^3_{18}}+2q\frac{f^3_6f^3_9}{f^7_3}+4q^2\frac{f^2_6f^3_{18}}{f^6_3}.
\end{equation}
\end{lemma}
\begin{lemma}\normalfont(Hirschhorn \cite[Eqs. (21.3.1), (22.1.4) and (39.2.8)]{H1}), we have
\begin{eqnarray}\label{3k}
f^3_1&=&f_3a(q^3)-3q f^3_9,\\\label{2k}
a(q)&=&a(q^3)+6q\frac{f^3_9}{f_3},\\\label{6k}
\frac{1}{f^3_1}&=&a^2(q^3)\frac{f^3_9}{f^{10}_3}+3qa(q^3)\frac{f^6_9}{f^{11}_3}+9q^2\frac{f^9_9}{f^{12}_3}.
\end{eqnarray}
where $$a(q)=\ds\sum^\infty_{m,n=-\infty}q^{m^2+mn+n^2}.$$
\end{lemma}
\begin{lemma} \label{1.2}\normalfont For $n\geq0$,  we have
\begin{equation}\label{12}
\ds\sum_{n=0}^{\infty}p_5(7n+3)q^n=10f_1^4 f_7+49qf_7^5.
\end{equation}
\end{lemma}
\begin{proof}
Setting $k=5$ in \eqref{1w}, we have
\begin{equation}\label{13}
\ds\sum_{n=0}^{\infty}p_5(n)q^n=f^5_1.
\end{equation}
In\cite[p. 303, Entry 17(v)]{Berndt}, we have
\begin{equation}\label{4}
f_1=f_{49}\left(\frac{B(q^7)}{C(q^7)}-\frac{A(q^7)}{B(q^7)}q-q^2+\frac{C(q^7)}{A(q^7)}q^5\right),
\end{equation}
where $A$, $B$ and $C$ are defined by\\
$\ds A=A(q):=\frac{f(-q^3,-q^4)}{f(-q^2)}$,~~~$\ds B=B(q):=\frac{f(-q^2,-q^5)}{f(-q^2)}$ and $\ds C=C:=\frac{f(-q,-q^6)}{f(-q^2)}$\\
Substituting \eqref{4} into \eqref{13}, we have
\begin{equation}\label{14}
\ds\sum_{n=0}^{\infty}p_5(n)q^n=f^5_{49}\left(\frac{B(q^7)}{C(q^7)}-\frac{A(q^7)}{B(q^7)}q-q^2+\frac{C(q^7)}{A(q^7)}q^5\right)^5.
\end{equation}
If we extract those terms in which the power of $q$ is congruent to 3 modulo 7, divide by $q^3$ and replace $q^7$ by $q$, we have
\begin{align}\label{15}\nonumber
\ds\sum_{n=0}^{\infty}p_5(7n+3)q^n&=f^5_7\left(20\left(\frac{A B^2}{C^3}+\frac{A^2 C q}{B^3}-\frac{B C^2 q^2}{A^3}\right)\right.\\
&\qquad\left.\vphantom{\ds\frac{8q^3}{S}}
-10 \left(\frac{A^3}{B C^2}-\frac{B^3 q}{A^2 C}-\frac{C^3 q^2}{A B^2}\right)-61q\right).
\end{align}
From \cite[P. 174, Entry 31]{Berndt1} and \cite[Eq. 3.11 and Eq. 3.15]{Berndt2} in the terms of $A$, $B$ and $C$, we have
\begin{align}
\frac{B^5}{A C^4}-\frac{A^5}{B^4 C}-\frac{C^5 q^3}{A^4 B}&= 3q,\label{8p}\\
\frac{A B^2}{C^3}+\frac{A^2 C q}{B^3}-\frac{B C^2 q^2}{A^3}&=\frac{f_1^4}{f_7^4}+8q,\label{9p}\\
\frac{A^3}{B C^2}-\frac{B^3 q}{A^2 C}-\frac{C^3 q^2}{A B^2}&=\frac{f_1^4}{f_7^4}+5 q,\label{10p}\\
\frac{B^7}{C^7}-\frac{A^7 q}{B^7}+\frac{C^7 q^5}{A^7}&=14\frac{f_1^4}{f_7^4}q+\frac{f_1^8}{f_7^8}+57 q^2\label{11p}.
\end{align}
Substituting  \eqref{9p} and \eqref{10p} into \eqref{15} and simplifying, we completed the Lemma \ref{1.2}.
\end{proof}

\begin{lemma}\label{0.2}\normalfont\cite[Theorem 3.2]{Berndt2} For $n\geq0$,  we have
\begin{equation}
\ds\sum_{n=0}^{\infty}p_7(7n)q^n=\frac{f^8_1}{f_7}+49qf^3_7f_1^4.
\end{equation}
\end{lemma}

\begin{lemma} \label{1.1}\normalfont For $n\geq0$,  we have
\begin{equation}\label{2}
\ds\sum_{n=0}^{\infty}p_9(7n+4)q^n=-90 f_1^8 f_7-882qf_1^4 f_7^5-2401q^2f_7^9.
\end{equation}
\end{lemma}
\begin{proof}
Setting $k=9$ in \eqref{1w}, we have
\begin{equation}\label{3}
\ds\sum_{n=0}^{\infty}p_9(n)q^n=f^9_1.
\end{equation}
Substituting \eqref{4} into \eqref{3}, we have
\begin{equation}\label{5}
\ds\sum_{n=0}^{\infty}p_9(n)q^n=f^9_{49}\left(\frac{B(q^7)}{C(q^7)}-\frac{A(q^7)}{B(q^7)}q-q^2+\frac{C(q^7)}{A(q^7)}q^5\right)^9.
\end{equation}
If we extract those terms in which the power of $q$ is congruent to 4 modulo 7, divide by $q^4$ and replace $q^7$ by $q$, we have
\begin{align}\label{6}\nonumber
\ds\sum_{n=0}^{\infty}p_9(7n+4)q^n&=f^9_7\left(\left(\frac{36 B^7}{C^7}-\frac{252 A^2 B^4}{C^6}+\frac{126 A^4 B}{C^5}\right)\right.\\ \nonumber
&\qquad\left.\vphantom{\ds\frac{8q^3}{S}}+q \left(-\frac{36 A^7}{B^7}+\frac{756 A^5}{B^4 C}-\frac{3780 A^3}{B C^2}-\frac{756 B^5}{A C^4}+\frac{4284 A B^2}{C^3}\right)\right.\\ \nonumber
&\qquad\left.\vphantom{\ds\frac{8q^3}{S}}+q^2\left(-\frac{252 A^4 C^2}{B^6}+\frac{4284 A^2 C}{B^3}+\frac{3780 B^3}{A^2 C}-9745\right)\right.\\ \nonumber
&\qquad\left.\vphantom{\ds\frac{8q^3}{S}}+q^3\left(\frac{126 B^4 C}{A^5}-\frac{4284 B C^2}{A^3}-\frac{126 A C^4}{B^5}+\frac{3780 C^3}{A B^2}\right)\right.\\
&\qquad\left.\vphantom{\ds\frac{8q^3}{S}}+q^4\left(\frac{756 C^5}{A^4 B}-\frac{252 B^2 C^4}{A^6}\right)+\frac{36C^7}{A^7}q^5\right).
\end{align}
Rearrange the above equation, we have
\begin{align}\label{6.1}\nonumber
\ds\sum_{n=0}^{\infty}p_9(7n+4)q^n&=f^9_7\left(36 \left(\frac{B^7}{C^7}-\frac{A^7 q}{B^7}+\frac{C^7 q^5}{A^7}\right)-252 \left(\frac{A^2 B^4}{C^6}-\frac{ A^4 C^2 q^2}{B^6}-\frac{B^2 C^4 q^4}{A^6}\right)\right.\\ \nonumber
&\qquad\left.\vphantom{\ds\frac{8q^3}{S}}
+126 \left(\frac{A^4 B}{C^5}+\frac{ B^4 C q^3}{A^5}-\frac{A C^4 q^3}{B^5}\right)-756 q \left(\frac{B^5}{A C^4}-\frac{A^5}{B^4 C}-\frac{C^5 q^3}{A^4 B}\right)\right.\\\nonumber
&\qquad\left.\vphantom{\ds\frac{8q^3}{S}}
+4284 q \left(\frac{A B^2}{C^3}+\frac{ A^2 C q}{B^3}-\frac{ B C^2 q^2}{A^3}\right)-3780 q \left(\frac{A^3}{B C^2}-\frac{B^3 q}{A^2 C}-\frac{C^3 q^2}{A B^2}\right)\right.\\
&\qquad\left.\vphantom{\ds\frac{8q^3}{S}}
-9745 q^2
\right).
\end{align}
From above follow that
\begin{align}\label{7}\nonumber
&\ds\sum_{n=0}^{\infty}p_9(7n+4)q^n=f^9_7\left(36\left(\frac{B^7}{C^7}-\frac{A^7 q}{B^7}+\frac{C^7 q^5}{A^7}\right)
\right.\\ \nonumber
&\qquad\left.\vphantom{\ds\frac{8q^3}{S}}
-252 \left(\left(\frac{A B^2}{C^3}+\frac{A^2Cq}{B^3}-\frac{B C^2q^2}{A^3}\right)^2-2q\left(\frac{A^3}{B C^2}-\frac{B^3 q}{A^2 C}-\frac{C^3 q^2}{A B^2}\right)\right)\right.\\ \nonumber
&\qquad\left.\vphantom{\ds\frac{8q^3}{S}}+126 \left(q \left(\frac{B^5}{A C^4}-\frac{A^5}{B^4 C}-\frac{C^5 q^3}{A^4 B}\right)
+\left(\frac{A B^2}{C^3}+\frac{ A^2C q}{B^3}-\frac{ B C^2q^2}{A^3}\right) \left(\frac{A^3}{B C^2}-\frac{B^3 q}{A^2 C}-\frac{C^3 q^2}{A B^2}\right)\right)\right.\\ \nonumber
&\qquad\left.\vphantom{\ds\frac{8q^3}{S}}-756q\left(\frac{B^5}{A C^4}-\frac{A^5}{B^4 C}-\frac{C^5 q^3}{A^4 B}\right)+4284 q \left(\frac{A B^2}{C^3}+\frac{A^2 C q}{B^3}-\frac{B C^2 q^2}{A^3}\right)\right.\\
&\qquad\left.\vphantom{\ds\frac{8q^3}{S}}
-3780 q \left(\frac{A^3}{B C^2}-\frac{B^3 q}{A^2 C}-\frac{C^3 q^2}{A B^2}\right)-9367 q^2\right).
\end{align}
Substituting \eqref{8p}, \eqref{9p}, \eqref{10p} and \eqref{11p} into \eqref{7}, we have
\begin{align}\label{7.1}\nonumber
&\ds\sum_{n=0}^{\infty}p_9(7n+4)q^n=f^9_7\left(36\left(\frac{f_1^8}{f_7^8}+14\frac{ f_1^4 q}{f_7^4}+57 q^2\right)-252 \left(\left(\frac{f_1^4}{f_7^4}+8 q\right){}^2-2 q \left(\frac{f_1^4}{f_7^4}+5 q\right)\right)\right.\\ \nonumber
&\qquad\left.\vphantom{\ds\frac{8q^3}{S}}
+126 \left((3 q) q+\left(\frac{f_1^4}{f_7^4}+5 q\right) \left(\frac{f_1^4}{f_7^4}+8 q\right)\right)-756 q (3 q)+4284 q \left(\frac{f_1^4}{f_7^4}+8 q\right)
\right.\\
&\qquad\left.\vphantom{\ds\frac{8q^3}{S}}
-3780 q \left(\frac{f_1^4}{f_7^4}+5 q\right)-9367 q^2
\right).
\end{align}
Simplify  the equation \eqref{7.1}, we completed the Lemma \ref{1.1}.
\end{proof}
\section{Congruence for $(2,15)$-regular bipartition}
\begin{theorem}\label{d1}\normalfont For $n\geq0$,  we have
\begin{eqnarray}
B_{2,15}(9n+8)&\equiv&0\pmod5\label{b.1},\\
B_{2,15}(27n+14)&\equiv&0\pmod5\label{b.2},\\
B_{2,15}(27n+23)&\equiv&2B_{2,15}(3n+2)\pmod5\label{b.3}.
\end{eqnarray}
\end{theorem}
\begin{proof}
Setting $s=2$ and $t=15$ in \eqref{0.1}, we have
\begin{equation}\label{b1}
\sum_{n=0}^{\infty}B_{2,15}(n)q^n=\frac{f_2f_{15}}{f^2_1}.
\end{equation}
Substituting \eqref{a5} into \eqref{b1}, we have
\begin{equation}\label{b2}
\sum_{n=0}^{\infty}B_{2,15}(n)q^n=f_{15}\left(\frac{f^4_6f^6_9}{f^8_3f^3_{18}}+2q\frac{f^3_6f^3_9}{f^7_3}+4q^2\frac{f^2_6f^3_{18}}{f^6_3}\right).
\end{equation}
If we extract those terms in which the power of $q$ is congruent to 2 modulo 3, divide by $q^2$ and replace $q^3$ by $q$, we have
\begin{equation}\label{b3}
\sum_{n=0}^{\infty}B_{2,15}(3n+2)q^n\equiv4\frac{f^2_2f^3_6}{f_1}\pmod5.
\end{equation}
Substituting \eqref{4w} into \eqref{b3}, we have
\begin{equation}\label{b4}
\sum_{n=0}^{\infty}B_{2,15}(3n+2)q^n\equiv4f^3_6\left(\frac{f_6 f_9^2}{f_3 f_{18}}+q\frac{f_{18}^2}{f_9}\right)\pmod5.
\end{equation}
If we extract those terms in which the power of $q$ is congruent to 1 modulo 3, divide by $q$ and replace $q^3$ by $q$, we have
\begin{equation}\label{b5}
\sum_{n=0}^{\infty}B_{2,15}(9n+5)q^n\equiv4\frac{f^2_6f^3_2}{f_3}\pmod5.
\end{equation}
Now replacing $q$ by $q^2$ in \eqref{3k}, we have
\begin{align}\label{b5.2}
f^3_2 &=f_{6}a(q^6)-3q^2f^3_{18},
\end{align}
Substituting \eqref{b5.2} in \eqref{b5} and extract those terms in which the power of $q$ is congruent to 2 modulo 3, divide by $q^2$ and replace $q^3$ by $q$, we have
\begin{equation}\label{b6}
\sum_{n=0}^{\infty}B_{2,15}(27n+23)q^n\equiv3\frac{f^2_2f^3_6}{f_1}\pmod5.
\end{equation}
This  completed  the  proof Theorem \ref{d1}, follow from  \eqref{b4}, \eqref{b6}.
Substituting \eqref{b5.2} in \eqref{b5} and extract those terms in which the power of $q$ is congruent to 1 modulo 3, divide by $q$ and replace $q^3$ by $q$, we have \eqref{b.2}.
\end{proof}
\subsection*{Proof Theorem \ref{y}}
Equation \eqref{y1} follow from \eqref{b.3} by the mathematical induction. Employing \eqref{b.1} and \eqref{b.2} in \eqref{y1} we obtain \eqref{y2} and \eqref{y3}.
\section{Congruence for $(7,11)$-regular bipartition}
\subsection*{Proof of Theorem \ref{a}}
Setting $s=7$ and $t=11$ in \eqref{0.1}, we have
\begin{eqnarray}\label{17}\nonumber
\sum_{n=0}^{\infty}B_{7,11}(n)q^n&\equiv&f_7f^9_1 \pmod{11},\\
&\equiv&f_7\ds\sum_{n=0}^\infty p_9(n)q^n\pmod{11}.
\end{eqnarray}
It follows that
\begin{equation}\label{17.2}
\sum_{n=0}^{\infty}B_{7,11}(7n+4)q^n\equiv f_1\ds\sum_{n=0}^{\infty}p_9(7n+4)q^n\pmod{11}.
\end{equation}
Now from \eqref{2} and \eqref{17.2}, we have
\begin{align}\label{18}\nonumber
\sum_{n=0}^{\infty}B_{7,11}(7n+4)q^n&\equiv f_1\left(9 f_1^8 f_7+9q f_1^4 f_7^5+8q^2 f_7^9\right),\\
&\equiv9 f_7 f_1^9+9q f_7^5 f_1^5+8q^2 f_7^9 f_1.
\end{align}
Substituting \eqref{13}, \eqref{4} and \eqref{3} into \eqref{18}, we have
\begin{align}\label{19}\nonumber
\sum_{n=0}^{\infty}B_{7,11}(7n+4)q^n&\equiv9 f_7\ds\sum_{n=0}^{\infty}p_9(n)q^n+9 f_7^5 \ds\sum_{n=0}^{\infty}p_5(n)q^{n+1}\\
&+8q^2 f_7^9 f_{49}\left(\frac{B(q^7)}{C(q^7)}-\frac{A(q^7)}{B(q^7)}q-q^2+\frac{C(q^7)}{A(q^7)}q^5\right).
\end{align}
It follows that
\begin{equation}\label{20}
\sum_{n=0}^{\infty}B_{7,11}(49n+32)q^n\equiv9 f_1\ds\sum_{n=0}^{\infty}p_9(7n+4)q^n+9 f_1^5 \ds\sum_{n=0}^{\infty}p_5(7n+3)q^{n}
+3 f_1^9 f_{7}.
\end{equation}
Substituting Lemma \ref{1.2} and Lemma \ref{1.1} into \eqref{20}, we have
\begin{align}\label{21}\nonumber
\sum_{n=0}^{\infty}B_{7,11}(49n+32)q^n&\equiv9f_1\left(9 f_1^8 f_7+9qf_1^4 f_7^5+8q^2 f_7^9\right)+9 f_1^5\left(10 f_1^4 f_7+5q f_7^5\right)+3f_1^9 f_7,\\
&\equiv9f_7 f_1^9+5qf_7^5 f_1^5+6q^2f_7^9 f_1.
\end{align}
Similarly, we find
\begin{eqnarray}\nonumber
\sum_{n=0}^{\infty}B_{7,11}(343n+228)q^n&\equiv&4f_7 f_1^9+7qf_7^5 f_1^5+6q^2f_7^9 f_1,\\\nonumber
\sum_{n=0}^{\infty}B_{7,11}(2401n+1600)q^n&\equiv&f_7 f_1^9+5qf_7^5 f_1^5+10q^2f_7^9 f_1,\\\nonumber
\sum_{n=0}^{\infty}B_{7,11}(16807n+11204)q^n&\equiv&5f_7 f_1^9+qf_7^5 f_1^5+8q^2f_7^9 f_1,\\\nonumber
\sum_{n=0}^{\infty}B_{7,11}(117649n+78432)q^n&\equiv&3f_7 f_1^9+6qf_7^5 f_1^5+7q^2f_7^9 f_1,\\\nonumber
\sum_{n=0}^{\infty}B_{7,11}(823543n+549028)q^n&\equiv&3f_7 f_1^9+2qf_7^5 f_1^5+2q^2f_7^9 f_1,\\\nonumber
\sum_{n=0}^{\infty}B_{7,11}(5764801n+3843200)q^n&\equiv&f_7 f_1^9+4qf_7^5 f_1^5+2q^2f_7^9 f_1,\\\nonumber
\sum_{n=0}^{\infty}B_{7,11}(40353607n+26902404)q^n&\equiv&3f_7 f_1^9+7qf_7^5 f_1^5+8q^2f_7^9 f_1,\\\nonumber
\sum_{n=0}^{\infty}B_{7,11}(282475249n+188316832)q^n&\equiv&f_7 f_1^9+7qf_7^5 f_1^5+2q^2f_7^9 f_1,\\\label{22}
\sum_{n=0}^{\infty}B_{7,11}(1977326743n+1318217828)q^n&\equiv&8q^2f_7^9 f_1.
\end{eqnarray}
Substituting \eqref{4} into \eqref{22}, we have 
\begin{equation}
 \sum_{n=0}^{\infty}B_{7,11}\left(7^{11}n+\frac{2(7^{11}-1)}{3}\right)q^n\equiv8q^2f_7^9 f_{49}\left(\frac{B(q^7)}{C(q^7)}-\frac{A(q^7)}{B(q^7)}q-q^2+\frac{C(q^7)}{A(q^7)}q^5\right).
\end{equation}
There are no terms on the right of $q^{7n+1}$, $q^{7n+5}$ and $q^{7n+6}$, so
\begin{eqnarray}\label{a4}
B_{7,11}\left(7^{11}(7n+k)+\frac{2(7^{11}-1)}{3}\right)&\equiv&0 \pmod{11},~\textrm{where}~k=1,5,6.\label{a4}
\end{eqnarray}
and 
\begin{equation}\label{a3}
B_{7,11}\left(7^{12}n+\frac{2(7^{12}-1)}{3}\right)\equiv3B_{7,11}(n)\pmod{11}.
\end{equation}
Equation \eqref{a1} follow from \eqref{a3} by mathematical induction. Employing \eqref{a1} in \eqref{a4}, we obtain \eqref{a2}.
\section{Congruence for $(27,11)$-regular bipartition}
\subsection*{Proof of Theorem \ref{b}}
Setting $s=27$ and $t=11$, we have
\begin{eqnarray}\nonumber
\sum_{n=0}^{\infty}B_{27,11}(n)q^n&=&\frac{f_{27}f_{11}}{f^2_1},\\\label{1}
 &\equiv&f_{27}f^9_1\pmod{11}.
 \end{eqnarray}
Substituting \eqref{3k} into \eqref{1}, we have
\begin{eqnarray}\nonumber
\sum_{n=0}^{\infty}B_{27,11}(n)q^n&\equiv&f_{27} \left(a(q^3) f_3-3 f_9^3 q\right)^3,\\\nonumber
&\equiv&a(q^3)^3 f_3^3 f_{27}+2q a(q^3)^2 f_3^2 f_{27} f_9^3 +5q^2a(q^3) f_3 f_{27} f_9^6 +6q^3f_{27} f_9^9.
\end{eqnarray}
It follows that
\begin{eqnarray}\label{9}
\sum_{n=0}^{\infty}B_{27,11}(3n)q^n&\equiv&a(q)^3 f_1^3 f_{9}+6qf_{9} f_3^9.
\end{eqnarray}
Substituting \eqref{3k} and \eqref{2k} into \eqref{9}, we have
\begin{eqnarray}\nonumber
\sum_{n=0}^{\infty}B_{27,11}(3n)q^n&\equiv&f_{9}\left(a(q^3)+6q\frac{f^3_9}{f_3}\right)^3\left(f_3a(q^3)-3q f^3_9\right)+6qf_{9} f_3^9,\\\nonumber
&\equiv&a(q^3)^4 f_3 f_9+6q f_3^9 f_9+4q a(q^3)^3 f_9^4+10q^2\frac{ a(q^3)^2 f_9^7}{f_3}+2q^3\frac{a(q^3) f_9^{10}}{f_3^2}+q^4\frac{f_9^{13}}{f_3^3}.
\end{eqnarray}
It follows that
\begin{eqnarray}\label{10}
\sum_{n=0}^{\infty}B_{27,11}(9n+3)q^n&\equiv&6f_1^9 f_3+4a(q)^3 f_3^4+q\frac{f_3^{13}}{f_1^3}.
\end{eqnarray}
Substituting \eqref{3k}, \eqref{2k} and \eqref{6k} into \eqref{10}, we have
\begin{eqnarray}\nonumber
\sum_{n=0}^{\infty}B_{27,11}(9n+3)q^n&\equiv&6f_3\left(f_3a(q^3)-3q f^3_9\right)^3+4f_3^4\left(a(q^3)+6q\frac{f^3_9}{f_3}\right)^3\\\nonumber
&&+qf_3^{13}\left(a^2(q^3)\frac{f^3_9}{f^{10}_3}+3qa(q^3)\frac{f^6_9}{f^{11}_3}+9q^2\frac{f^9_9}{f^{12}_3}\right),\\
&\equiv&10 a(q^3)^3 f_3^4+8q a(q^3)^2 f_3^3 f_9^3 +3q^2 a(q^3) f_3^2 f_9^6+7q^3 f_3 f_9^9.
\end{eqnarray}
It follows that
\begin{eqnarray}\label{10k}
\sum_{n=0}^{\infty}B_{27,11}(27n+12)q^n&\equiv&8a(q)^2 f_1^3 f_3^3.
\end{eqnarray}
Substituting \eqref{3k} and \eqref{2k} into \eqref{10k}, we have
\begin{eqnarray}\nonumber
\sum_{n=0}^{\infty}B_{27,11}(27n+12)q^n&\equiv&8f_3^3\left(a(q^3)+6q\frac{f^3_9}{f_3}\right)^2\left(f_3a(q^3)-3q f^3_9\right),\\
&\equiv&8 a(q^3)^3 f_3^4+6qa(q^3)^2 f_3^3 f_9^3 +5q^3 f_3 f_9^9.
\end{eqnarray}
It follows that
\begin{equation} \label{11}
B_{27,11}(81n+66)\equiv0\pmod{11}
\end{equation}
and
\begin{equation}\label{12}
B_{27,11}(81n+39)\equiv9B_{27,11}(27n+12)\pmod{11}.
\end{equation}
This completes the proof of Theorem \ref{b} follow from \eqref{11} and  \eqref{12}.
\section{Congruence for $(243,17)$-regular bipartition}
\subsection*{Proof of Theorem \ref{c}}
Setting $s=243$ and $t=17$, we have
\begin{eqnarray}\nonumber
\sum_{n=0}^{\infty}B_{243,17}(n)q^n&=&\frac{f_{243}f_{17}}{f^2_1},\\\label{g1}
 &\equiv&f_{243}f^{15}_1\pmod{17}.
\end{eqnarray}
Substituting \eqref{3k} into \eqref{g1}, we have
\begin{eqnarray}\nonumber
\sum_{n=0}^{\infty}B_{243,17}(n)q^n&\equiv&f_{243} \left(a(q^3) f_3-3 f_9^3 q\right)^5,\\\nonumber
&\equiv&a(q^3)^5 f_3^5 f_{243}+2q a(q^3)^4 f_3^4 f_9^3 f_{243}+5q^2 a(q^3)^3 f_3^3 f_9^6 f_{243}+2q^3 a(q^3)^2 f_3^2 f_9^9 f_{243}\\\nonumber
&&+14q^4a(q^3) f_3 f_9^{12} f_{243}+12q^5 f_9^{15} f_{243}.
\end{eqnarray}
It follows that
\begin{equation}\label{g2}
\sum_{n=0}^{\infty}B_{243,17}(3n+2)q^n\equiv5a(q)^3 f_1^3 f_3^6f_{81} +12q f_3^{15}f_{81}.
\end{equation}
Substituting \eqref{3k} and \eqref{2k} into \eqref{g2}, we have
\begin{eqnarray}\nonumber
\sum_{n=0}^{\infty}B_{243,17}(3n+2)q^n&\equiv&5f_3^6f_{81}\left(a(q^3)+6q\frac{f^3_9}{f_3}\right)^3\left(f_3a(q^3)-3q f^3_9\right) +12q f_3^{15}f_{81}\\\nonumber
&\equiv&5 a(q^3)^4 f_3^7 f_{81}+12qf_3^{15} f_{81}+7q a(q^3)^3 f_3^6 f_9^3 f_{81}+15q^2 a(q^3)^2 f_3^5 f_9^6 f_{81}\\\nonumber
&&+4q^3a(q^3) f_3^4 f_9^9 f_{81}+7q^4 f_3^3 f_9^{12} f_{81}.
\end{eqnarray}
If we extract those terms in which the power of $q$ is congruent to 1 modulo 3, divide by $q$ and replace $q^3$ by $q$, we have
\begin{equation}\label{g3}
\sum_{n=0}^{\infty}B_{243,17}(9n+5)q^n\equiv12 f_1^{15} f_{27} +7 a(q)^3 f_1^6 f_3^3 f_{27} +7q f_1^3 f_3^{12} f_{27}.
\end{equation}
Substituting \eqref{3k} and \eqref{2k} into \eqref{g3}, we have
\begin{eqnarray}\nonumber
\sum_{n=0}^{\infty}B_{243,17}(9n+5)q^n&\equiv&12 \left(a(q^3) f_3-3 f_9^3 q\right)^5 f_{27} +7 \left(a(q^3)+6q\frac{f^3_9}{f_3}\right)^3 \left(a(q^3) f_3-3 f_9^3 q\right)^2 f_3^3 f_{27}\\\nonumber
&&+7q \left(a(q^3) f_3-3 f_9^3 q\right) f_3^{12} f_{27},\\\nonumber
&\equiv&2 a(q^3)^5 f_3^5 f_{27}+7q a(q^3) f_3^{13} f_{27}+6q a(q^3)^4 f_3^4 f_9^3 f_{27}+13q^2 f_3^{12} f_9^3 f_{27}\\\nonumber
&&+4q^2 a(q^3)^3 f_3^3 f_9^6 f_{27}+4q^3 a(q^3)^2 f_3^2 f_9^9 f_{27}+8q^4 a(q^3) f_3 f_9^{12} f_{27}+16q^5 f_9^{15} f_{27}.
\end{eqnarray}
If we extract those terms in which the power of $q$ is congruent to 2 modulo 3, divide by $q$ and replace $q^3$ by $q$, we have
\begin{eqnarray}\label{g4}
 \sum_{n=0}^{\infty}B_{243,17}(27n+23)q^n&\equiv&13 f_1^{12} f_3^3 f_9+4 a(q)^3 f_1^3 f_3^6 f_9+16q f_3^{15} f_9.
\end{eqnarray}
Substituting \eqref{3k} and \eqref{2k} into \eqref{g4}, we have
\begin{eqnarray}\nonumber
\sum_{n=0}^{\infty}B_{243,17}(27n+23)q^n&\equiv&13 \left(a(q^3) f_3-3 f_9^3 q\right)^4 f_3^3 f_9+4 \left(a(q^3)+6q\frac{f^3_9}{f_3}\right)^3 \left(a(q^3) f_3-3 f_9^3 q\right) f_3^6 f_9\\\nonumber
&&+16q f_3^{15} f_9,\\\label{g5}
&\equiv&16q f_3^{15} f_9+6q a(q^3)^3 f_3^6 f_9^4+8q^4 f_3^3 f_9^{13}.
\end{eqnarray}
This completes the proof of Theorem \ref{c} follow from \eqref{g5}.

\end{document}